\documentclass[11pt,a4paper,english]{article}
\usepackage{amssymb}
\usepackage{amsfonts,yfonts}
\usepackage[all]{xy}
\usepackage{amsfonts,amscd,amssymb,amsthm}
\usepackage{amsfonts,amssymb,amsmath,amsxtra,amscd,amsthm,eucal,graphicx,graphics}
\usepackage{tikz}
\usepackage{enumerate}
 \usepackage[T1]{fontenc}
\usepackage[utf8]{inputenc}
\usepackage{babel}

\newcommand{\cupdot}{\mathbin{\mathaccent\cdot\cup}}

\theoremstyle{definition}
\newtheorem{theorem}{Theorem}[section]

\newtheorem{proposition}[theorem]{Proposition}
\newtheorem{lemma}[theorem]{Lemma}

\newtheorem{conjecture}[theorem]{Conjecture}

 \title{On the maximum number of colorings of a  graph}
    \author{Aysel Erey\footnote{e-mail: aysel.erey@gmail.com}}
    \date{Department of Mathematics\\ University of Denver \\Denver, CO, USA 80208 \\[\baselineskip] \today }

\begin{document}

\maketitle

\begin{abstract}
Let $\mathcal{C}_k(n)$ be the family of all connected $k$-chromatic graphs of order $n$. Given a natural number $x\geq k$, we consider the problem of finding the maximum number of $x$-colorings among graphs in  $\mathcal{C}_k(n)$. When $k\leq 3$ the answer to this problem is known, and when $k\geq 4$ the problem is wide open. For $k\geq 4$ it was conjectured that  the maximum number of $x$-colorings is $x(x-1)\cdots (x-k+1)\,x^{n-k}$. In this article, we prove this conjecture under the  additional condition that the independence number of the graphs is at most $2$.
\end{abstract}




\section{Introduction}
All graphs in this article are {\it simple}, that is, they do not have loops or multiple edges. Let  $V(G)$ and $E(G)$ be the vertex set and edge set of a graph $G$, respectively. The {\em order} of $G$ is $|V(G)|$ which is denoted by $n_G$, and the {\em size} of $G$ is $|E(G)|$. For a nonnegative integer $x$, an \textit{$x$-coloring} of $G$ is a function $f:V(G)\rightarrow \{1,\dots , x\}$ such that $f(u)\neq f(v)$ for every $uv\in E(G)$. The \textit{chromatic number} $\chi(G)$ is smallest $x$ for which $G$ has an $x$-coloring and $G$ is called \textit{k-chromatic} if $\chi(G)=k$. Let $\pi(G,x)$ denote the {\em chromatic polynomial of G}. For nonnegative integers $x$, $\pi(G,x)$ counts the number of $x$-colorings of $G$.

There has been a great interest in maximizing or minimizing the number of $x$-colorings over various families of graphs. Here we shall focus on the family of all connected graphs with fixed chromatic number and fixed order. Let $\mathcal{C}_k(n)$ be the family of all connected $k$-chromatic graphs of order $n$.
Given a natural number $x\geq k$, we consider the problem of finding the maximum number of $x$-colorings among graphs in  $\mathcal{C}_k(n)$. When $k\leq 3$ the answer to this problem is known. It is well known that (see, for example, \cite{dongbook})  for $k=2$ and $x\geq 2$, the maximum number of $x$-colorings of a graph in $\mathcal{C}_2(n)$ is equal to $x(x-1)^{n-1}$, and extremal graphs are trees when $x\geq 3$. Also, for $x\geq k=3$, the maximum number of $x$-colorings of a graph in  $\mathcal{C}_3(n)$ is
$$(x-1)^{n}-(x-1) \ \ \ \text{for odd} \ n$$
and
$$ (x-1)^{n}-(x-1)^2 \ \ \ \text{for even} \ n$$
 and furthermore the extremal graph is the odd cycle $C_{n}$ when $n$ is odd and odd cycle with a vertex of degree $1$ attached to the cycle (denoted $C_{n-1}^1$) when $n$ is even \cite{tomescu}. For $k\geq 4$, the problem is wide open. For $k\geq 4$, Tomescu~\cite{tomescu} (see also \cite{dongbook,tomescufrench}) conjectured that the maximum number of $x$-colorings of a graph in $\mathcal{C}_k(n)$ is $(x)_{\downarrow k}(x-1)^{n-k}=x(x-1)\cdots (x-k+1)(x-1)^{n-k}$, and  the extremal graphs are those which belong to the family of all connected $k$-chromatic graphs of order $n$ with clique number $k$ and size ${k\choose 2}+n-k$, denoted by $\mathcal{C}^*_k(n)$.

 \begin{conjecture}\cite[pg. 315]{dongbook} \label{tomesdongconj} Let $G$ be a graph in  $\mathcal{C}_k(n)$  where $k\geq 4$. Then for every $x\in \mathbb{N}$ with $x\geq k$
 	$$\pi(G,x)\leq (x)_{\downarrow k}(x-1)^{n-k}.$$
 	Moreover, the equality holds if and only if $G$ belongs to  $\mathcal{C}^*_k(n)$.
 \end{conjecture}

Several authors studied this conjecture. Tomescu~\cite{tomescu} proved this conjecture for $k=4$ under the additional condition that graphs  are planar. In \cite{brownerey}, the authors proved this conjecture for every $k\geq 4$, provided that $x\geq n-2+\left( {n\choose 2}-{k\choose 2}-n+k \right) ^2$. Our main result in this article is Theorem~\ref{main} which proves this conjecture for graphs whose independence numbers are at most 2 (i.e. complements of triangle-free graphs).
 
Let $G / e$ be the graph formed from $G$ by {\it contracting} edge $e$, that is, by identifying the ends of $e$ (and taking the underlying simple graph). For $e\notin E(G)$, observe that $$\chi(G)=\operatorname{min}\{\chi(G+e)\, , \, \chi(G/ e)\}$$ and the well known  \textit{Edge Addition-Contraction Formula} says that   $$\pi(G,x)=\pi(G+e,x)+\pi(G/ e,x).$$ Also, the chromatic polynomial of a graph can be computed by using the {\em Complete Cut-set Theorem}: If $G_1$ and $G_2$ are two graphs such that $G_1\cap G_2\cong K_r$, then
$$\pi(G_1\cup G_2,x)=\frac{\pi(G_1,x)\, \pi(G_2,x)}{(x)_{\downarrow r}}.$$ 

Let $G\cupdot H$ be the disjoint union of $G$ and $H$, and $G\vee H$ be their join. It is easy to see that $$\pi(G\vee K_1,x)=x\, \pi(G,x-1).$$
The maximum degree of a graph $G$ is $\Delta(G)$, and a vertex $v$ of $G$ is {\it universal} if it is joined to all other vertices. In \cite{brownerey}, Conjecture~\ref{tomesdongconj} was proven for graphs which contain a universal vertex.

\begin{lemma}\label{universallemma}\cite{brownerey}
	Let $G\in \mathcal{C}_k(n)$ and $\Delta(G)=n-1$. Then, for all $x\in \mathbb{N}$ with $x\geq k$, the inequality  $\pi(G,x)\leq(x)_{\downarrow k}\, (x-1)^{n-k}$ holds.  Furthermore, the equality is achieved if and only if $G=K_1\vee (K_{k-1}\cupdot (n-k)K_1)$. 
\end{lemma}

Lastly, let $\omega(G)$ and $\alpha(G)$ be the clique number and independence number of $G$ respectively.

\section{Main Results}

\begin{lemma}\label{general_k_clique_lemma}
	Let $G\in \mathcal{C}_k(n)$ and $\omega(G)=k$. Then for all $x\in \mathbb{N}$ with $x\geq k$, $$\pi(G,x)\leq (x)_{\downarrow k}\, (x-1)^{n-k}$$ with equality if and only if $G\in  \mathcal{C}_k^*(n) $.
\end{lemma}
\begin{proof}
	Let $H$ be a $k$-clique of $G$. If $G$ has no cycle $C$ such that $E(C)\setminus E(H)\neq \emptyset$ then  $G\in  \mathcal{C}_k^*(n) $ and the result is clear. So we assume that there exists a cycle $C$ of $G$ such that $E(C)\setminus E(H)\neq \emptyset$ (i.e. $G\notin  \mathcal{C}_k^*(n)$).  We may choose a cycle $C$ such that $|E(C)\cap E(H)|\leq 1$, as $H$ is a clique. Let $G'$ be  a minimal spanning connected subgraph of $G$ which contains $H$ and $C$. First we shall show that 	$\pi(G',x) = (x-1)_{\downarrow k-1}\, \pi(C,x)\, (x-1)^{n-k-n_C+1}$. If $|E(C)\cap E(H)|=0$ (resp. $|E(C)\cap E(H)|=1$), let $G_1$ and $G_2$ be two subgraphs of $G'$ such that $G_1$ contains $H$, $G_2$ contains $C$, $G_1\cup G_2=G'$, and $G_1$ and $G_2$ intersect in a single vertex (resp. edge) of $H$. By the Complete Cut-set Theorem, if $|E(C)\cap E(H)|=0$ then $\pi(G',x)=\frac{\pi(G_1,x)\pi(G_2,x)}{x}$, and if $|E(C)\cap E(H)|=1$ then $\pi(G',x)=\frac{\pi(G_1,x)\pi(G_2,x)}{x(x-1)}$. In each case, $G_1\in \mathcal{C}_k^*(n_{G_1})$ and $G_2$ is a connected unicyclic graph. Therefore,
	$$\pi(G_1,x)=(x)_{\downarrow k} (x-1)^{n_{G_1}-k}$$
	and
	$$\pi(G_2,x)=\pi(C,x) (x-1)^{n_{G_2}-n_C}.$$
	If $|E(C)\cap E(H)|=0$ then $n_{G_1}+n_{G_2}=n+1$, and if $|E(C)\cap E(H)|=1$ then $n_{G_1}+n_{G_2}=n+2$. Thus, we obtain $\pi(G',x) = (x-1)_{\downarrow k-1}\, \pi(C,x)\, (x-1)^{n-k-n_C+1}$. Also,
	
	\begin{eqnarray*}
\pi(G',x) &=& (x-1)_{\downarrow k-1}\, \pi(C,x)\, (x-1)^{n-k-n_C+1}\\
&=& (x-1)_{\downarrow k-1}\, \left((x-1)^{n_C}+(-1)^{n_C}(x-1) \right)\, (x-1)^{n-k-n_C+1}\\
&=&  (x-1)_{\downarrow k-1}\, \left( (x-1)^{n-k+1}+(-1)^{n_C}(x-1)^{n-k-n_C+2} \right)\\
&<&  (x-1)_{\downarrow k-1}\, \left( (x-1)^{n-k+1}+(x-1)^{n-k} \right)\\
&=&  (x)_{\downarrow k}\,  (x-1)^{n-k} 
	\end{eqnarray*}
	where the inequality holds as $n_C\geq 3$. Now the result follows since $\pi(G',x)\geq \pi(G,x)$.

\end{proof}

A \textit{cut-set} of a connected graph is a subset of the vertex set whose removal disconnects the graph. To prove our main result, we first deal with graphs which have a cut-set of size at most $2$.

\begin{proposition}\label{cutset2prop}Let $G$ be a connected $k$-chromatic graph with $\alpha(G)=2$. If $G$ has a stable cut-set $S$ of size at most $2$ then 
\begin{itemize}
\item [(i)] $G\setminus S$ has exactly two connected components, say, $G_1$ and $G_2$,
\item [(ii)] $G_1$ and $G_2$ are complete graphs,
\item [(iii)] $\operatorname{max}\{\chi(G_1),\chi(G_2)\}\geq k-1$,
\item [(iv)] For every $u$ in $S$, either $V(G_1)\subseteq N_G(u)$ or $V(G_2)\subseteq N_G(u)$.
\end{itemize}
\end{proposition}

\begin{proof}
\begin{itemize}
\item[(i)] If $G\setminus S$ had more than two components then we could pick a vertex from each component and get a stable set of size at least $3$. And this would contradict with the assumption that $\alpha(G)=2$.

\item[(ii)] Suppose on the contrary that $G_1$ or $G_2$ is not a complete graph. Without loss, we may assume $G_1$ has two nonadjacent vertices $u$ and $v$. Let $w$ be a vertex of $G_2$. Then $\{u,v,w\}$ is a stable set of size $3$ and again this  contradicts with $\alpha(G)=2$.

\item[(iii)] Suppose that $\chi(G_1)$ and  $\chi(G_2)$ are at most $k-2$. Then we can properly color $G_1$ and $G_2$ with colors $1,\dots , k-2$ and we can assign a new color $k-1$ to all vertices in $S$. This yields a proper $(k-1)$-coloring of $G$ and this contradicts with the assumption that $G$ is $k$-chromatic.

\item[(iv)] If there exists a vertex $u$ in $S$ such that $u$ has a non-neighbor $v$ in $G_1$ and a non-neighbor $w$ in $G_2$ then we get a stable set $\{u,v,w\}$ of size $3$ and this contradicts with $\alpha(G)=2$.
\end{itemize}
\end{proof}
Note that if $G\in \mathcal{C}_{k}^*(n)$ and $\alpha(G)=2$ then either $G$ is a $k$-clique with a path of size one hanging off a vertex of the clique (denoted by $F_{1,k}$) or  $G$ is a $k$-clique with a path of size two hanging off a vertex of the clique (denoted by $F_{2,k}$).

\begin{lemma}\label{cutvertexlemma}
	Let $G\in \mathcal{C}_k(n)$ with $\alpha(G)=2$. Let $x\in \mathbb{N}$ with $x\geq k$ and $u$ be a cut-vertex of $G$. Then, $\pi(G,x)\leq (x)_{\downarrow k}\, (x-1)^{n-k}$. Furthermore, the equality holds if and only if $G\cong F_{1,k}$ or $G\cong F_{2,k}$.
\end{lemma}
\begin{proof}
	By Proposition \ref{cutset2prop}, $G-u$ has exactly two connected components and they are complete graphs. Now it is easy to see that  $G$ is chordal and hence $\omega(G)=k$. Thus, the result follows by Lemma~\ref{general_k_clique_lemma}.
	
\end{proof}

\begin{lemma}\label{cutsetlemma}Let $G$ be a graph in $\mathcal{C}_k(n)$ with $\alpha(G)=2$ and $k\geq 4$. If $G$ has a stable cut-set of size $2$ then
$$\pi(G,x)\leq (x)_{\downarrow k}\, (x-1)^{n-k}$$ for all $x\in \mathbb{N}$ with $x\geq k$. Furthermore,  the equality is achieved if and only if $G\cong F_{2,k}$.
\end{lemma}
\begin{proof}
	 Let $S=\{u,v\}$ be a stable cut-set of $G$.
If $\omega(G)=k$ then the result follows from Lemma~\ref{general_k_clique_lemma}, so we may assume that $\omega(G)<k$.
By Proposition~\ref{cutset2prop}, the graph $G\setminus S$ has exactly two connected components, say $G_1$ and $G_2$, and we may assume $G_1\cong K_p$,  $G_2\cong K_q$ where $p\geq q$. Now, $p\geq k-1$ by Proposition~\ref{cutset2prop} and $\omega(G)<k$ by the assumption. Therefore, $p=k-1$. Since $\omega(G)<k$, every vertex in $S$ has at least one non-neighbor in $G_1$. Let $u'$ and $v'$ be two vertices of $G_1$ which are non-neighbors of $u$ and $v$ respectively.


Since $V(G_1)\nsubseteq N_G(u)$ and $V(G_1)\nsubseteq N_G(v)$, all vertices in $S$ are adjacent to all vertices in $G_2$ by Proposition~\ref{cutset2prop}. The graph $G_2$ has at most $k-2$ vertices, as $\omega(G)<k$. If $G_2$ has less than $k-2$ vertices then we can find a proper $k-1$ coloring $c$ of $G$ (we can first properly color the vertices of $G_1$ with colors $1,2,\dots k-1$ and assign  $c(u')$ (resp. $c(v')$) to $u$ (resp. $v$) and then we can properly color the vertices of $G_2$ with colors $\{1,2,\dots k-1\}\setminus \{c(u), c(v)\}$ which yields a proper $k-1$ coloring of $G$). Therefore $G_2$ has exactly $k-2$ vertices and $q=k-2$.

Since $\alpha(G)=2$, the vertices $u$ and $v$ have no common non-neighbor. Therefore,
$$G / uv \cong K_1 \vee (K_{k-1}\cupdot K_{k-2}).$$
Now it is easy to see that 
\begin{eqnarray}\label{contraction}
\pi(G/ uv,x )=(x-1)_{\downarrow k-1}\, (x)_{\downarrow k-1}.
\end{eqnarray}
Let $H_1$ (resp. $H_2$) be the subgraph of $G+uv$ induced by the vertex set $ V(G_1)\cup S$ (resp. $V(G_2)\cup S$). Now, the graphs $H_1$ and $H_2$ intersect at the edge $uv$ in $G+uv$. Therefore,
$$\pi(G+uv,x)=\frac{\pi(H_1,x)\, \pi(H_2,x)}{x(x-1)}.$$
Since $H_2\cong K_{k}$, we get $\pi(H_2,x)=(x)_{\downarrow k}$. Also, one of the vertices of $S$ has a neighbor in $G_1$, as $G$ is connected. So, $H_1$ contains a spanning subgraph which is isomorphic to a graph in $\mathcal{C}^*_{k-1}(k+1)$. Thus, $\pi(H_1,x)\leq (x)_{\downarrow k-1}(x-1)^2$. Now,
\begin{eqnarray}\label{addition}
\pi(G+uv,x)\leq \frac{(x)_{\downarrow k}\, (x)_{\downarrow k-1}\,(x-1)^2 }{x(x-1)}=(x-1)\,(x)_{\downarrow k-1}\, (x-1)_{\downarrow k-1}.
\end{eqnarray}
Using the edge addition-contraction formula and \eqref{contraction} and \eqref{addition} we get
\begin{eqnarray*}
\pi(G,x) &=& \pi(G+uv,x)+\pi(G/ uv,x)\\
&\leq & (x-1)\,(x)_{\downarrow k-1}\, (x-1)_{\downarrow k-1}\,+\, (x-1)_{\downarrow k-1}\, (x)_{\downarrow k-1}\\
&=& (x)_{\downarrow k}\, (x)_{\downarrow k-1}.
\end{eqnarray*}
 The graph $G$ has $2k-1$ vertices, so $(x)_{\downarrow k\,}(x-1)^{n-k}=(x)_{\downarrow k}\,(x-1)^{k-1}$. Now it is clear that  
 $$(x)_{\downarrow k}\, (x)_{\downarrow k-1} < (x)_{\downarrow k}\,(x-1)^{k-1} $$
 holds for $k\geq 4$, as $(x)_{\downarrow k-1}=x(x-1)(x-2)\cdots$ and $x(x-2)<(x-1)^2$.
 
 \end{proof}

\begin{theorem}\label{main}Let $G$ be a graph in $\mathcal{C}_k(n)$ with $\alpha(G)\leq 2$ and $k\geq 4$. Then, for every $x\in \mathbb{N}$ with $x\geq k$,
	$$\pi(G,x)\leq (x)_{\downarrow k}\, (x-1)^{n-k}.$$  Furthermore, the equality is achieved if and only if $G\cong F_{1,k}$, $G\cong F_{2,k}$ or $k=n$.
\end{theorem}
\begin{proof}
Since $\alpha(G)\chi(G)\geq n$, the equality $k=4$ implies $n\leq 8$. Computations show that the result holds to be true when $n\leq 8$. So we may assume that $k\geq 5$.
	We proceed by induction on the number of vertices. For the basis step, $n = k$ and $G$ is a complete graph. Hence, $\pi(G,x)=(x)_{\downarrow k}$ and  now the result is clear.

	Now we may assume that $G$ is a $k$-chromatic graph of order at least $k+1$. By Lemma~\ref{cutvertexlemma} and Lemma~\ref{cutsetlemma}, we may assume that $G$ has no stable cut-set of size at most $2$. Also, if  $\Delta(G)=n-1$ then the result follows by Lemma~\ref{universallemma}. Hence, we shall assume that $\Delta(G)<n-1$. Let $u$ be a vertex of maximum degree.  Set $t=n-1-\Delta(G)$ and let $\{v_1,\dots , v_t\}$ be the set of non-neighbors of $u$ in $G$,\ (that is, $\{v_1,\dots , v_t\}=V(G)\setminus N_G[u]$).
	We set $G_0 = G$ and
	$$G_i=G_{i-1}+uv_i$$
	$$H_i=G_i/ uv_i$$
	for $i=1,\dots , t$.  By applying the Edge Addition-Contraction Formula successively,
	\begin{equation}\label{mainsum}
	\pi(G,x)=\pi(G_t,x)+\sum_{i=1}^t\pi(H_i,x).
	\end{equation}

	Note that $k\leq \chi(G_t) \leq k+1$ and $k\leq \chi(H_i)\leq k+1$ for $i = 1,2,\ldots,t$. Since $u$ is a universal vertex of $G_t$, we have
	 \begin{equation}\label{G_t_univers}
	 \pi(G_t,x)=x\,\pi(G-u,x-1).
	 \end{equation}

Clearly, $\alpha(G-u)\leq 2$. Also, $G-u$ is connected as $G$ has no cut-vertex by the assumption. So, by the induction hypothesis, $$\pi(G-u,x)\leq (x)_{\downarrow \chi(G-u)}(x-1)^{n-1-\chi(G-u)}.$$
	 Now replacing $x$ with $x-1$ in the latter, we get
	 $$\pi(G-u,x-1)\leq (x-1)_{\downarrow \chi(G-u)}(x-2)^{n-1-\chi(G-u)}.$$
 Note that $k-1\leq \chi(G-u)\leq k$. Also, $(x-1)_{\downarrow k}(x-2)^{n-1-k}<(x-1)_{\downarrow k-1}(x-2)^{n-k}$. Therefore,
	$$\pi (G-u,x-1)\leq (x-1)_{\downarrow k-1}(x-2)^{n-k}.$$
	Since $(x)_{\downarrow k}=x(x-1)_{\downarrow k-1}$, by \eqref{G_t_univers} we obtain that
	\begin{equation}\label{G_tbound}
	\pi(G_t,x)\leq (x)_{\downarrow k}(x-2)^{n-k}.
	\end{equation}
	
	Now we shall give an upper bound for $\pi (H_i,x)$ for all $i$. Observe that $$H_i\cong K_1 \vee (G-\{u,v_i\})$$
	because $\alpha(G)=2$ and hence every vertex in $G-\{u,v_i\}$ is adjacent to either $u$ or $v_i$ in $G$. Therefore,
	\begin{equation}\label{H_i}
	\pi(H_i,x)=x\,\pi(G-\{u,v_i\},x-1).
	\end{equation}
	
	It is clear that $\alpha(G-\{u,v_i\})\leq 2$. Since $G$ has no stable cut-set of size $2$, the graph $G-\{u,v_i\}$ is connected. Also, $k-1\leq \chi(G-\{u,v_i\})\leq k$, as $u$ and $v_i$ are nonadjacent in $G$ and $\chi(G)=k$. By the induction hypothesis, 
	$$\pi(G-\{u,v_i\},x)\leq (x)_{\downarrow \chi(G-\{u,v_i\})}(x-1)^{n-2-\chi(G-\{u,v_i\})}.$$
	 Now replacing $x$ with $x-1$ in the latter, we get
	$$\pi(G-\{u,v_i\},x-1)\leq (x-1)_{\downarrow \chi(G-\{u,v_i\})}(x-2)^{n-2-\chi(G-\{u,v_i\})}.$$
 Observe that $(x-1)_{\downarrow k}(x-2)^{n-k-2}<(x-1)_{\downarrow k-1}(x-2)^{n-k-1}$. Thus,
	$$\pi(G-\{u,v_i\},x-1)\leq (x-1)_{\downarrow k-1}(x-2)^{n-k-1}.$$
	Since $(x)_{\downarrow k}=x(x-1)_{\downarrow k-1}$, by \eqref{H_i} we obtain that
	\begin{equation}\label{H_i_bound}
	\pi(H_i,x)\leq (x)_{\downarrow k}(x-2)^{n-k-1}.
	\end{equation}
	By \eqref{mainsum}, \eqref{G_tbound} and \eqref{H_i_bound}, we get
	\begin{eqnarray*}
	\pi(G,x)&\leq & (x)_{\downarrow k}(x-2)^{n-k}+(n-1-\Delta(G))(x)_{\downarrow k}(x-2)^{n-k-1}\\
	&=& (x)_{\downarrow k}(x-2)^{n-k-1}(x-3+n-\Delta(G)).
	\end{eqnarray*}
	Now, it suffices to show that $(x-2)^{n-k-1}(x-3+n-\Delta(G))\leq (x-1)^{n-k}$. The graph $G$ is neither a complete graph nor an odd cycle, so $\Delta(G)\geq k$ by Brook's Theorem.  Hence, $n-\Delta(G)\leq n-k$. Now,
	\begin{eqnarray*}
	(x-3+n-\Delta(G))(x-2)^{n-k-1} &\leq & (x-3+n-k)\,(x-2)^{n-k-1}\\
	&=&(x-2-1+n-k)\,(x-2)^{n-k-1}\\
	&=&(x-2)^{n-k}-(x-2)^{n-k-1}+(n-k)(x-2)^{n-k-1}\\
	&<&(x-2)^{n-k}+(n-k)(x-2)^{n-k-1}\\
	&\leq&(x-2+1)^{n-k}\\
	&=& (x-1)^{n-k}
	\end{eqnarray*}
	where the last inequality holds, as $$(x-2+1)^{n-k}=(x-2)^{n-k}+(n-k)(x-2)^{n-k-1}+{n-k\choose 2}(x-2)^{n-k-2}+\cdots .$$
	Thus, $\pi(G,x)\leq (x)_{\downarrow k}(x-1)^{n-k}$ and the result follows.
\end{proof}

\vskip0.4in

\bibliographystyle{elsarticle-num}

\end{document}